\documentclass[12pt]{elsarticle} %
\usepackage{geometry}                
\usepackage[parfill]{parskip}    
\usepackage{graphicx}
\usepackage{amsmath}
\usepackage{amssymb}
\usepackage{amsthm}
\usepackage{epstopdf}
\journal{arXiv}

\newtheorem{Definition}{Definition}
\newtheorem{Theorem}{Theorem}
\newtheorem{Lemma}[Theorem]{Lemma}
\newtheorem{Corollary}[Theorem]{Corollary}

\begin{document}
\begin{frontmatter}
\title{Investigation of generalised Fredholm equations and their direct solution subject to constraints}
\author[newcastle]{P.G. Clark\corref{cor1}}
\ead{peter.clark@newcastle.ac.uk}
\author[brad]{A.S. Wood}
\ead{a.s.wood@bradford.ac.uk}
\author[brad]{P. Olley}
\ead{p.olley@bradford.ac.uk}
\address[newcastle]{Newcastle University, Biogerontology Building, Newcastle upon Tyne, NE4 5PL, UK}
\address[brad]{University of Bradford, Richmond Road, Bradford, BD7 1DP, UK}
\cortext[cor1]{Corresponding author}
\begin{abstract}
This paper explores the solution of Fredholm-like equations with infinite dimensional solution spaces. We set out to find a method for determining a particular solution to a Fredholm-like equation subject to a given constraint. The relevance and application comes via a connection to certain dynamics of gas-like systems.
\end{abstract}
\end{frontmatter}

\section{Fredholm Equations}
Certain generalisations of the Bolzmann equation have been developed to study the mesoscopic dynamics of systems that are much more complex than simple gases \cite{Bellouquid:2006rr, Grmela:1983xy, Grmela:1983jk}. In a recent study of one such system \cite{Clark:2012fk} linked to polymer dynamics in micro-moulding, an application of the Chapman-Enskog expansion was implemented. This method produces a linear integral equation, itself a generalisation of the Fredholm equation of the second kind, that is usually solved by \textit{ad hoc} methods, in part by appealing to constraints that appear to restrict the equation in a countably infinite number of degrees of freedom. A conclusion that may be reached is that the unconstrained Fredholm-like equation has an infinite dimensional solution space. Whether or not this is the case, no general solution method for Fredholm-like equations with infinite dimensional solution spaces was available and the \textit{ad hoc} method usually used was not applicable in our case. Here we develop a method for determining a solution to a Fredholm-like equation subject to a given constraint with out considering the determination of the complete solution space.

Fredholm studied integral equations with finite and constant limits of integration  in one free and one integrated variable \cite{Fredholm:1903yq}, essentially the one dimensional case:
\[\varphi(x)+\int_0^1f(x,y)\varphi(y)dy=\phi(x)\]
However, modern Fredholm theory is now couched and understood in terms of the study of Fredholm operators on Hilbert and Banach spaces \cite{Ruston:1986fk}, which cover a wider range of equations. It is normally studied with emphasis placed upon the spectra of the equation or operator. Equations that are not Fredholm are often referred to as singular, even if they have a Fredholm-like form. For example, the equation
\[\phi(x)-\sqrt{\frac{2}{\pi}}\int_0^{\infty}\sin(xs)\phi(s)ds=0\]
is known to have infinitely many linearly independent solutions, which is a situation that cannot occur in Fredholm equations.

Boundary value problems for differential equations may be reduced to the problem of solving a Fredholm equation \cite[p.62]{Jerri:1985fk}. Further, the mixed boundary value problem may gives rise to a system of singular Fredholm-like equations \cite[p.68]{Jerri:1985fk}. Consequently, both Fredholm and singular Fredholm-like equations occur commonly in mathematical physics, including the study of electromagnetic fields. Additionally, singular integral equations arise in Hilbert's problem \cite[p.298]{Mikhlin:1964uq}, in problems involving elastic membranes in contact \cite[p.310]{Mikhlin:1964uq}, and in certain problems of fluid dynamics \cite[p.325]{Mikhlin:1964uq}.

\section{The Bott-Duffin Pseudo-inverse Applied to Operator Equations}
A common approach to solving integral equations involves expressing the equations as operator equations on a Hilbert space, and a quite common general form for such equations is
\begin{equation}
\label{Fredholm Operator 1}
\mathcal{A}\vec{x}+\vec{\varphi}=\vec{x}
\end{equation}
For certain conditions on the operator \(\mathcal{A}\), equation \eqref{Fredholm Operator 1} has a unique solution that may be obtained by a Neumann series
\begin{equation}
\label{Neumann series 1}
\sum_{i=0}^{\infty}\mathcal{A}^i\vec{\varphi}=\vec{x}
\end{equation}
However, sometimes it will occur that, despite there being no unique solution, we would like to seek a particular solution subject to a (certain) set of constraints
\begin{equation}
\label{Constraints 1}
\left<\vec{x},\vec{y}_i\right>=0, \: i=1,\cdots ,n
\end{equation}
We shall assume that \(\left<\vec{y}_i,\vec{y}_j\right>=\delta_{ij}\) noting that if this were not the case it is easy to construct an equivalent set of \(\vec{y}_i\) with that property using the Gram-Schmidt process. These constraints define a subspace of the Hilbert space, itself a Hilbert space.
\newline
\begin{Definition} \label{proj def 1}
We define a projection-like operator into this space given by
\begin{equation*}
\mathcal{P}_k\vec{x}=\vec{x}-\sum_{i=1}^n\vec{k}_i\left<\vec{x},\vec{y}_i\right>, \: \left<\vec{k}_i,\vec{y}_j\right>=\delta_{i,j}
\end{equation*}
\end{Definition}
We call this a projection-like operator because it has two key properties of a projection. Firstly, it acts as an identity on the subspace (it can be clearly seen that equation \eqref{Constraints 1} implies \(\mathcal{P}_k\vec{x}=\vec{x}\)). Secondly, it projects into the subspace, which is easily proven by applying equation \eqref{Constraints 1} to definition \ref{proj def 1} giving
\[\left<\mathcal{P}_k\vec{x},\vec{y}_j\right>=\left<\vec{x},\vec{y}_j\right>-\sum_{i=1}^n\left<\vec{k}_i,\vec{y}_j\right>\left<\vec{x},\vec{y}_i\right>=\left<\vec{x},\vec{y}_j\right>-\sum_{i=1}^n\delta_{i,j}\left<\vec{x},\vec{y}_i\right>=0\]
\begin{Definition} \label{comp proj def 1}
We define the complimentary projection \(\tilde{\mathcal{P}}_k\) which has the two equivalent definitions
\begin{equation*}
\tilde{\mathcal{P}}_k+\mathcal{P}_k=\mathcal{I} \Leftrightarrow \tilde{\mathcal{P}}_k\vec{x}=\sum_{i=1}^n\vec{k}_i\left<\vec{x},\vec{y}_i\right>
\end{equation*}
\(\mathcal{I}\) is the identity operator.
\end{Definition}
\quad
\begin{Theorem} \label{theorem of solution}
If \(\tilde{\mathcal{P}}_k\sum_{i=0}^{\infty}(\mathcal{A}\mathcal{P}_k)^i\vec{\varphi}=\vec{0}\) then \(\sum_{i=0}^{\infty}(\mathcal{A}\mathcal{P}_k)^i\vec{\varphi}\) is a solution to equation \eqref{Fredholm Operator 1} subject to constraints \eqref{Constraints 1}, provided that the vectors \(\vec{k}_i\) are linearly independent and \(\left<\vec{k}_i,\vec{y}_j\right>=\delta_{i,j}\).
\end{Theorem}
\begin{proof}
Given the constructions in definitions \ref{proj def 1} and \ref{comp proj def 1}, and that equation \eqref{Fredholm Operator 1} can be reformulated as \(\vec{\varphi}=(\mathcal{I}-\mathcal{A})\vec{x}\), we may consider the two forms of the related operator equation
\begin{equation}
\label{Fredholm Operator 2}
\vec{\varphi}=(\mathcal{I}-\mathcal{A})\mathcal{P}_k\vec{x}+\tilde{\mathcal{P}}_k\vec{x}\Leftrightarrow \mathcal{A}\mathcal{P}_k\vec{x}+\vec{\varphi}=\mathcal{P}_k\vec{x}+\tilde{\mathcal{P}}_k\vec{x}=\mathcal{I}\vec{x}=\vec{x}
\end{equation}
The validity of the relation is due to the first definition of \(\tilde{\mathcal{P}}_k\) in definition \ref{comp proj def 1}. It is trivially proved that equation \eqref{Fredholm Operator 1} and equation \eqref{Fredholm Operator 2} have identical solutions when \(\tilde{\mathcal{P}}_k\vec{x}=\vec{0}\), and provided that the linear independence of all \(\vec{k}_i\) is assured then this condition is equivalent to equation \eqref{Constraints 1}. Consequently, if the operator \(\mathcal{A}\mathcal{P}_k\) meets specified necessary conditions, then equation \eqref{Fredholm Operator 2} may be solved by the Neumann series
\begin{equation*}
\sum_{i=0}^{\infty}(\mathcal{A}\mathcal{P}_k)^i\vec{\varphi}=\vec{x}
\end{equation*}
If \(\tilde{\mathcal{P}}_k\sum_{i=0}^{\infty}(\mathcal{A}\mathcal{P}_k)^i\vec{\varphi}=\vec{0}\), then this is also a solution to equation \eqref{Fredholm Operator 1} subject to constraints \eqref{Constraints 1}.
\end{proof}
This method is comparable to the Bott-Duffin method for solving certain matrix equations except that the `projection' operator proposed here is more general. This allows us to vary the \(\vec{k}_i\), subject to conditions of linear independence and \(\left<\vec{k}_i,\vec{y}_j\right>=\delta_{i,j}\), seeking values for which \(\tilde{\mathcal{P}}_k\sum_{i=0}^{\infty}(\mathcal{A}\mathcal{P}_k)^i\vec{\varphi}\) is well defined and is equal to \(\vec{0}\) in the pursuit of solving equation \eqref{Fredholm Operator 1} subject to constraints \eqref{Constraints 1}.

\section{Lemmas Relating to Conditions on Products of Hilbert Spaces}
It is useful in the application of the previous results to define some new notation and derive some lemmas.
\newline
\begin{Definition}
Given an innerproduct \(\left<,\right>:\mathbf{H}_1\otimes \mathbf{H}_2\times\mathbf{H}_1\otimes \mathbf{H}_2\rightarrow \mathbb{C}\) of a Hilbert space given by the tensor product of Hilbert spaces \(\mathbf{H}_1\) and \(\mathbf{H}_2\), and given that \(\left<,\right>_i:\mathbf{H}_i\times\mathbf{H}_i\rightarrow \mathbb{C}\) is the innerproduct of \(\mathbf{H}_i\) (\(i=1,2\)), we use \(\left<,\right>_{2^{\prime}}:\mathbf{H}_1\otimes \mathbf{H}_2\times\mathbf{H}_2\rightarrow \mathbf{H}_1\) to represent a bi-linear operator such that 
\[\left<\left<\vec{x},\vec{y}\,\right>_{2^{\prime}},\vec{z}\,\right>_1=\left<\vec{x},\vec{z}\otimes\vec{y}\,\right>\]
\end{Definition}
assuming that such an operator exists. 
\newline
\begin{Lemma}
If \(\mathbf{H}_1\) is a separable Hilbert space then it follows that \(\vec{0}=\left<\vec{x},\vec{y}_i\right>_{2^{\prime}}\) provides a countably infinite set of constraints like those in equation \eqref{Constraints 1} with the solution space being orthogonal to the vectors \(\vec{\psi}_j\otimes\vec{y}_i\) where \(\vec{\psi}_i\) is an orthonormal basis of \(\mathbf{H}_1\).
\end{Lemma}
\begin{proof}
\[\vec{0}=\left<\vec{x},\vec{y}_i\right>_{2^{\prime}}\Leftrightarrow 0=\left<\vec{0},\vec{\psi}_j\right>_1=\left<\left<\vec{x},\vec{y}_i\right>_{2^{\prime}},\vec{\psi}_j\right>_1=\left<\vec{x},\vec{\psi}_j\otimes\vec{y}_i\right> \: \forall \: j\in \mathbb{N}\]
\end{proof}
\begin{Lemma}
The orthonormality of these constraints is guaranteed by the condition \(\left<\vec{y}_i,\vec{y}_j\right>_2=\delta_{ij}\)
\end{Lemma}
\begin{proof}
\[\left<\vec{\psi}_r\otimes \vec{y}_i,\vec{\psi}_s\otimes\vec{y}_j\right>=\left<\vec{\psi}_r,\vec{\psi}_s\right>_1 \left<\vec{y}_i,\vec{y}_j\right>_2=\delta_{rs} \left<\vec{y}_i,\vec{y}_j\right>_2\]
\end{proof}

\begin{Lemma} \label{exp proj operator}
if \(\mathbf{H}_1\) is a separable Hilbert space, it follows that
\[\mathcal{P}_k\vec{x} =\vec{x}-\sum_{i=1}^{n}\left<\vec{x},\vec{y}_i\right>_{2^{\prime}}\otimes \vec{k}_i\]
is equivalent to the definition \ref{proj def 1}.
\end{Lemma}
\begin{proof}
\begin{align*}
\mathcal{P}_k\vec{x} & =\vec{x}-\sum_{i=1}^{n}\left<\vec{x},\vec{y}_i\right>_{2^{\prime}}\otimes \vec{k}_i \nonumber\\
 & =\vec{x}-\sum_{i=1}^{n}\sum_{j=0}^{\infty}\left<\left<\vec{x},\vec{y}_i\right>_{2^{\prime}},\vec{\psi}_j\right>_1\vec{\psi}_j\otimes \vec{k}_i \nonumber\\
 & =\vec{x}-\sum_{i=1}^{n}\sum_{j=0}^{\infty}\vec{\psi}_j\otimes \vec{k}_i\left<\vec{x},\vec{\psi}_j\otimes\vec{y}_i\right>
\end{align*}
\end{proof}

\begin{Lemma}
If \(\left<\vec{k}_i,\vec{y}_j\right>_2=\delta_{ij}\) then the condition on \(\vec{k}_i\) in definition \ref{proj def 1} is also satisfied
\end{Lemma}
\begin{proof}
\[\left<\vec{\psi}_j\otimes \vec{k}_i,\vec{\psi}_s\otimes\vec{y}_r\right>=\left<\vec{\psi}_j,\vec{\psi}_s\right>_1 \left<\vec{k}_i,\vec{y}_r\right>_2=\delta_{js}\left<\vec{k}_i,\vec{y}_r\right>_2\]
\end{proof}

\begin{Lemma}
Given the orthonormal basis \(\vec{\psi}_i\) for the separable Hilbert space, it follows that the vectors \(\vec{\psi}_j\otimes \vec{k}_i\) are linearly independent provided that the vectors \(\vec{k}_i\) are linearly independent.
\end{Lemma}
\begin{proof}
\[\vec{0}=\sum_{i=1}^{n}\sum_{j=0}^{\infty}\alpha_{ij}\vec{\psi}_j\otimes \vec{k}_i=\sum_{j=0}^{\infty}\vec{\psi}_j\otimes \left(\sum_{i=1}^{n}\alpha_{ij}\vec{k}_i\right)\Rightarrow \sum_{i=1}^{n}\alpha_{ij}\vec{k}_i=\vec{0}\]
\end{proof}

\section{Double Sum Reordering}
In order to manipulate infinite series of operators to obtain further proofs it is necessary to generalise notions of absolute convergence and investigate the reordering of double sums.
\newline
\begin{Definition}
We define \(\lim_{i,j\rightarrow \infty}a_{ij}=a\) if for all \(\varepsilon>0\) there exists \(\mu(\varepsilon)\) such that for all \(i,j>\mu(\varepsilon)\) we have \(\left|a-a_{ij}\right|<\varepsilon\). We generalise this to linear operators (\(\lim_{i,j\rightarrow \infty}\mathcal{X}_{i,j}=\mathcal{X}\)) with the variant condition \(\left\|\mathcal{X}-\mathcal{X}_{i,j}\right\|<\varepsilon\).
\end{Definition}
\quad
\begin{Theorem} \label{sum reordering theorem}
Assuming that \(\mathcal{X}_{i,j}\) is linear then the convergence of
\[\lim_{i,j\rightarrow \infty}\sum_{i=0}^{n}\sum_{j=0}^{m} \left\|\mathcal{X}_{i,j}\right\|\]
implies
\[\sum_{i=0}^{\infty}\sum_{j=0}^{\infty}\mathcal{X}_{i,j}=\sum_{j=0}^{\infty}\sum_{i=0}^{\infty}\mathcal{X}_{i,j}=\sum_{i=0}^{\infty}\mathcal{X}_{\sigma(i)}\]
where \(\sigma\) is an arbitrary bijection from the natural numbers (\(0\) to \(\infty\)) onto all pairs of natural numbers. These expressions all converge and in a generalised sense are absolutely convergent.
\end{Theorem}
\begin{proof}
\[\sum_{i=0}^n\sum_{j=0}^m\left|\mathcal{X}_{i,j}\vec{x}\right|=\left|\vec{x}\right|\sum_{i=0}^n\sum_{j=0}^m\left|\mathcal{X}_{i,j}\hat{x}\right|\leqslant\left|\vec{x}\right|\sum_{i=0}^n\sum_{j=0}^m\left\|\mathcal{X}_{i,j}\right\|\]
This implies that, for all \(\vec{x}\), \(\lim_{i,j\rightarrow \infty}\sum_{i=0}^{n}\sum_{j=0}^{m}\left|\mathcal{X}_{i,j}\vec{x}\right|\) is convergent if \linebreak\(\lim_{i,j\rightarrow \infty}\sum_{i=0}^{n}\allowbreak\sum_{j=0}^{m} \allowbreak\left\|\mathcal{X}_{i,j}\right\|\) is. It is well known that, in such cases,
\[\sum_{i=0}^{\infty}\sum_{j=0}^{\infty}\mathcal{X}_{i,j}\vec{x}=\sum_{j=0}^{\infty}\sum_{i=0}^{\infty}\mathcal{X}_{i,j}\vec{x}=\sum_{i=0}^{\infty}\mathcal{X}_{\sigma(i)}\vec{x}\]
are all absolutely convergent. However, since \(\vec{x}\) is arbitrary this implies that
\begin{equation}
\sum_{i=0}^{\infty}\sum_{j=0}^{\infty}\mathcal{X}_{i,j}=\sum_{j=0}^{\infty}\sum_{i=0}^{\infty}\mathcal{X}_{i,j}=\sum_{i=0}^{\infty}\mathcal{X}_{\sigma(i)}\label{sum reordering}
\end{equation}
Since, by assumption, \(\sum_{i=0}^{\infty}\sum_{j=0}^{\infty}\left\|\mathcal{X}_{i,j}\right\|<\infty\), this expression's summation terms may be reordered in the same way, thus demonstrating that the reordered sums \eqref{sum reordering} also possess this generalised notion of absolute convergence.
\end{proof}
\begin{Corollary} \label{cauchy ordering corol}
Assuming that \(\mathcal{X}_{i,j}\) is linear, then the identity
\[\sum_{i=0}^{\infty}\sum_{j=0}^{i}\mathcal{X}_{j,i-j}=\sum_{i=0}^{\infty}\sum_{j=0}^{\infty}\mathcal{X}_{i,j}=\sum_{j=0}^{\infty}\sum_{i=0}^{\infty}\mathcal{X}_{i,j}\]
is valid if any one of its forms is absolutely convergent which implies they all are.
\end{Corollary}
\begin{proof}
If we take
\begin{align*}
  \sigma(i)=&\left(i-\frac{1}{2}\left(\sqrt{\frac{8i+1}{4}}-\frac{1}{2}\right)^2-\frac{1}{2}\left\lfloor\sqrt{\frac{8i+1}{4}}-\frac{1}{2}\right\rfloor,\right.    \\
    &  \left.\frac{3}{2}\left\lfloor\sqrt{\frac{8i+1}{4}}-\frac{1}{2}\right\rfloor-i+\frac{1}{2}\left(\sqrt{\frac{8i+1}{4}}-\frac{1}{2}\right)^2\right)
\end{align*}
we then obtain (from theorem \ref{sum reordering theorem}) the identity
\[\sum_{i=0}^{\infty}\sum_{j=0}^{i}\mathcal{X}_{j,i-j}=\sum_{i=0}^{\infty}\mathcal{X}_{\sigma(i)}=\sum_{i=0}^{\infty}\sum_{j=0}^{\infty}\mathcal{X}_{i,j}=\sum_{j=0}^{\infty}\sum_{i=0}^{\infty}\mathcal{X}_{i,j}\]
 The absolute convergence of one form implies the validity of the identity and the absolute convergence of the other forms is observed by noting the validity of the identity
\[\sum_{i=0}^{\infty}\sum_{j=0}^{i}\left\|\mathcal{X}_{j,i-j}\right\|=\sum_{i=0}^{\infty}\left\|\mathcal{X}_{\sigma(i)}\right\|=\sum_{i=0}^{\infty}\sum_{j=0}^{\infty}\left\|\mathcal{X}_{i,j}\right\|=\sum_{j=0}^{\infty}\sum_{i=0}^{\infty}\left\|\mathcal{X}_{i,j}\right\|\]
as all expressions are sums of real non negative numbers.
\end{proof}

\begin{Lemma} \label{Cauchy Product lemma}
If \(\lim_{n\rightarrow\infty}\sum_{i=0}^{n}\left\|\mathcal{X}_{i}\right\|<\infty\) and \(\lim_{m\rightarrow\infty}\sum_{j=0}^{m}\left\|\mathcal{Y}_j\right\|<\infty\) then \linebreak\(\lim_{n,m\rightarrow \infty}\sum_{i=0}^{n}\sum_{j=0}^{m}\left\|\mathcal{X}_{i}\mathcal{Y}_j\right\|\) exists and we have the identity
\begin{equation*}
\sum_{i=0}^{\infty}\sum_{j=0}^{i}\mathcal{X}_{j}\mathcal{Y}_{i-j}=\sum_{i=0}^{\infty}\sum_{j=0}^{\infty}\mathcal{X}_{i}\mathcal{Y}_j
\end{equation*}
This reordering is also absolutely convergent in the sense that we have defined.
\end{Lemma}
\begin{proof}
Considering
\[\sum_{i=0}^{n}\sum_{j=0}^{m}\left\|\mathcal{X}_{i}\mathcal{Y}_j\right\|\leqslant \sum_{i=0}^{n}\sum_{j=0}^{m}\left\|\mathcal{X}_{i}\right\|\left\|\mathcal{Y}_j\right\|=\sum_{i=0}^{n}\left\|\mathcal{X}_{i}\right\|\sum_{j=0}^{m}\left\|\mathcal{Y}_j\right\|,\]
it is clear that if \(\lim_{n\rightarrow\infty}\sum_{i=0}^{n}\left\|\mathcal{X}_{i}\right\|<\infty\) and \(\lim_{m\rightarrow\infty}\sum_{j=0}^{m}\left\|\mathcal{Y}_j\right\|<\infty\) then \linebreak\(\lim_{n,m\rightarrow \infty}\sum_{i=0}^{n}\sum_{j=0}^{m}\left\|\mathcal{X}_{i}\mathcal{Y}_j\right\|\) converges. We then simply appeal to Corollary \ref{cauchy ordering corol}.
\end{proof}

\section{Further Lemmas}
It is expedient in examining perturbation-like expansions on \(k\) to derive some lemmas about the convergence and equivalence of various infinite sums of operators on Hilbert spaces.
\newline
\begin{Lemma} \label{operator equivalence lemma}
Let \(\mathcal{X}\) and \(\mathcal{Y}\) be linear operators on a Hilbert space and let
 \begin{equation*}
\left\| \mathcal{X}\right\|<1 \Leftrightarrow \sum_{i=0}^{\infty}\left\| \mathcal{X}\right\|^i=\frac{1}{1-\left\| \mathcal{X}\right\|}<\infty .
\end{equation*}
Let us denote \(\sum_{i=0}^{\infty} \mathcal{X}^i=\overline{\mathcal{X}}\). If we define \(\sum_{i=0}^{\infty}\left(\mathcal{X}+ \mathcal{Y}\right)^i= \mathcal{Z}\) and \(\overline{\mathcal{X}}\sum_{i=0}^{\infty}\left(\mathcal{Y}\overline{\mathcal{X}}\right)^i=\overline{\mathcal{Z}}\) then when \(\left\| \mathcal{Y}\right\|<1-\left\| \mathcal{X}\right\|\) both \(\mathcal{Z}\) and \(\overline{\mathcal{Z}}\) converge and \(\mathcal{Z}=\overline{\mathcal{Z}}\)
\end{Lemma}
\begin{proof}
First observe that
\[\left\| \mathcal{Y}\right\|<1-\left\| \mathcal{X}\right\|\Rightarrow\left\| \mathcal{X}+ \mathcal{Y}\right\|\leqslant \left\| \mathcal{X}\right\|+\left\| \mathcal{Y}\right\|<1\]
implies that \(\mathcal{Z}\) converges. Also observe that
\[\left\| \mathcal{Y}\right\|<1-\left\| \mathcal{X}\right\|\Rightarrow\left\| \mathcal{Y}\overline{\mathcal{X}}\right\|\leqslant\left\| \mathcal{Y}\right\|\left\|\overline{\mathcal{X}}\right\|\leqslant\left\| \mathcal{Y}\right\|\sum_{i=0}^{\infty}\left\| \mathcal{X}\right\|^i=\frac{\left\| \mathcal{Y}\right\|}{1-\left\| \mathcal{X}\right\|}<1\]
implies that \(\overline{\mathcal{Z}}\) converges. Hence \(\left\| \mathcal{Y}\right\|<1-\left\| \mathcal{X}\right\|\) implies that both \(\mathcal{Z}\) and \(\overline{\mathcal{Z}}\) converge, as well as ensuring the convergence of \(\sum_{i=0}^{\infty} \mathcal{Y}^i=\overline{\mathcal{Y}}\).

By virtue of lemma \ref{Cauchy Product lemma}, and assuming that \(\left\| \mathcal{Y}\right\|<1-\left\| \mathcal{X}\right\|\), it follows that
\[\overline{\mathcal{X}} \mathcal{Y}\overline{\mathcal{X}}=\sum_{i=0}^{\infty} \mathcal{X}^i \mathcal{Y}\sum_{i^{\prime}=0}^{\infty} \mathcal{X}^{i^{\prime}}=\sum_{i=0}^{\infty}\sum_{j=0}^{i} \mathcal{X}^{i-j} \mathcal{Y} \mathcal{X}^{j}\]
converges absolutely. Assuming the absolute convergence of \(\sum_{i=0}^{\infty}\sum_{j_0+\cdots+j_n=i} \mathcal{X}^{j_0}\allowbreak\prod_{k=1}^{n} \mathcal{Y} \mathcal{X}^{j_k}\), and again using lemma \ref{Cauchy Product lemma}, it follows that
\begin{align*}
\left(\sum_{i=0}^{\infty}\sum_{j_0+\cdots+j_n=i} \mathcal{X}^{j_0}\prod_{k=1}^{n} \mathcal{Y} \mathcal{X}^{j_k}\right) \mathcal{Y}\sum_{i^{\prime}=0}^{\infty} \mathcal{X}^{i^{\prime}} & =\sum_{i=0}^{\infty}\sum_{j_0+\cdots+j_n=i} \mathcal{X}^{j_0}\left(\prod_{k=1}^{n} \mathcal{Y} \mathcal{X}^{j_k}\right) \mathcal{Y}\sum_{i^{\prime}=0}^{\infty} \mathcal{X}^{i^{\prime}} \\
& \hspace{-4em} =\sum_{i=0}^{\infty}\sum_{j_{n+1}=0}^{i}\sum_{j_0+\cdots+j_n=i-j_{n+1}}\mathcal{X}^{j_0}\left(\prod_{k=1}^{n} \mathcal{Y} \mathcal{X}^{j_k}\right) \mathcal{Y} \mathcal{X}^{j_{n+1}} \\
& =\sum_{i=0}^{\infty}\sum_{j_0+\cdots+j_n+j_{n+1}=i} \mathcal{X}^{j_0}\left(\prod_{k=1}^{n+1} \mathcal{Y} \mathcal{X}^{j_k}\right)
\end{align*}
converges absolutely. An immediate result of this is
\begin{align*}
\overline{\mathcal{Z}} & =\overline{\mathcal{X}}\sum_{i=0}^{\infty}\left(\mathcal{Y}\overline{\mathcal{X}}\right)^i  \nonumber\\
    &  =\sum_{i=0}^{\infty}\overline{\mathcal{X}}\left(\mathcal{Y}\overline{\mathcal{X}}\right)^i \nonumber\\
    & =\sum_{i=0}^{\infty}\sum_{i^{\prime}=0}^{\infty}\sum_{j_0+\cdots+j_i=i^{\prime}} \mathcal{X}^{j_0}\left(\prod_{k=1}^i \mathcal{Y} \mathcal{X}^{j_k}\right) \nonumber\\
    & =\sum_{i^{\prime\prime}=0}^{\infty}\sum_{j^{\prime}=0}^{i^{\prime\prime}}\sum_{j_0+\cdots+j_i=i^{\prime\prime}-j^{\prime}} \mathcal{X}^{j_0}\left(\prod_{k=1}^{j^{\prime}} \mathcal{Y} \mathcal{X}^{j_k}\right) \nonumber\\
    & =\sum_{i^{\prime\prime}=0}^{\infty}\left(\mathcal{X}+ \mathcal{Y}\right)^{i^{\prime\prime}} \nonumber\\
    &  = \mathcal{Z}
\end{align*}
where we have used theorem \ref{sum reordering theorem}. Thus, when \(\left\| \mathcal{Y}\right\|<1-\left\| \mathcal{X}\right\|\), \(\mathcal{Z}\) and \(\overline{\mathcal{Z}}\) are equivalent operators.
\end{proof}

\section{Regions of \(k\) Giving Solutions}
Given a solution \(k\) to \(\tilde{\mathcal{P}}_k\sum_{i=0}^{\infty}(\mathcal{A}\mathcal{P}_k)^i\vec{\varphi}=\vec{0}\), we seek to examine the properties of alternative values for \(k^{\prime}\) in the neighbourhood of \(k\).
\newline
\begin{Definition} \label{k prime def}
We define \(\vec{k}^{\prime}_i=\vec{k}_i+\varepsilon_{\eta}\vec{\eta}_i\) subject to \(\left|\vec{\eta}_i\right|=1\) and \(\left<\vec{\eta}_i,\vec{y}_j\right>_2=\vec{0}\) for all \(i\) and \(j\).
\end{Definition}
\quad
\begin{Lemma} \label{operator split lemma}
It follows from definition \ref{k prime def} and lemma \ref{exp proj operator} directly 
\[\mathcal{P}_{k^{\prime}}\vec{x}=\mathcal{P}_k\vec{x}-\varepsilon_{\eta}\tilde{\mathcal{P}}_{\eta}\vec{x}\]
\end{Lemma}
\begin{proof}
\begin{align*}
\mathcal{P}_{k^{\prime}}\vec{x} & =\vec{x}-\sum_{i=1}^{n}\left<\vec{x},\vec{y}_i\right>_{2^{\prime}}\otimes \left(\vec{k}_i+\varepsilon_{\eta}\vec{\eta}_i\right) \nonumber\\
    & =\vec{x}-\sum_{i=1}^{n}\left<\vec{x},\vec{y}_i\right>_{2^{\prime}}\otimes \vec{k}_i-\varepsilon_{\eta}\sum_{i=1}^{n}\left<\vec{x},\vec{y}_i\right>_{2^{\prime}}\otimes\vec{\eta}_i \nonumber\\
    & =\mathcal{P}_k\vec{x}-\varepsilon_{\eta}\tilde{\mathcal{P}}_{\eta}\vec{x}
    \end{align*}
\end{proof}

\begin{Definition}
We define \(\mathcal{B}_{k}=\sum_{i=0}^{\infty}\left(\mathcal{A}\mathcal{P}_k\right)^i\)
\end{Definition}
\quad
\begin{Theorem} \label{solution disk theorem}
Assuming \(\tilde{\mathcal{P}}_{k}\sum_{i=0}^{\infty}\left(\mathcal{A}\mathcal{P}_{k}\right)^i\vec{\varphi}=\vec{0}\) and \(\left\|\mathcal{A}\mathcal{P}_k\right\|+\left\|\varepsilon_{\eta}\mathcal{A}\tilde{\mathcal{P}}_{\eta}\right\|<1\), and the linear independence of the \(\vec{k}_i\) it follows that \(\tilde{\mathcal{P}}_{k^{\prime}}\sum_{i=0}^{\infty}\left(\mathcal{A}\mathcal{P}_{k^{\prime}}\right)^i\vec{\varphi}=\vec{0}\)
\end{Theorem} 
\begin{proof}
Note that, by virtue of lemma \ref{operator equivalence lemma}, on applying lemma \ref{operator split lemma} and assuming that all \(\vec{k}_i\) are linearly independent, we have
\begin{align*}
  \tilde{\mathcal{P}}_{k^{\prime}}\sum_{i=0}^{\infty}\left(\mathcal{A}\mathcal{P}_{k^{\prime}}\right)^i\vec{\varphi}  & =\tilde{\mathcal{P}}_{k^{\prime}}\sum_{i=0}^{\infty}\left(\mathcal{A}\mathcal{P}_k-\varepsilon_{\eta}\mathcal{A}\tilde{\mathcal{P}}_{\eta}\right)^i\vec{\varphi}  \nonumber\\
    &  =\tilde{\mathcal{P}}_{k^{\prime}}\mathcal{B}_{k}\sum_{i=0}^{\infty}\left(-\varepsilon_{\eta}\mathcal{A}\tilde{\mathcal{P}}_{\eta}\mathcal{B}_{k}\right)^i\vec{\varphi} \nonumber\\
    & =\vec{0}
\end{align*}
by virtue of \(\tilde{\mathcal{P}}_k\sum_{i=0}^{\infty}(\mathcal{A}\mathcal{P}_k)^i\vec{\varphi}=\vec{0}\) implying that \(\tilde{\mathcal{P}}_{k^{\prime}}\mathcal{B}_{k}\vec{\varphi}=\vec{0}\) and \(\tilde{\mathcal{P}}_{\eta}\mathcal{B}_{k}\vec{\varphi}=\vec{0}\). A sufficient condition for the convergence of this sum is given by \(\left\|\mathcal{A}\mathcal{P}_k\right\|+\left\|\varepsilon_{\eta}\mathcal{A}\tilde{\mathcal{P}}_{\eta}\right\|<1\), assuming \(\left\|\mathcal{A}\mathcal{P}_k\right\|<1\) (which is a necessary condition of our initial assumption).
\end{proof}
\begin{Corollary}
There is an open region around \(k\) where \(k^{\prime}\) also gives solutions of \(\tilde{\mathcal{P}}_{k^{\prime}}\sum_{i=0}^{\infty}\left(\mathcal{A}\mathcal{P}_{k^{\prime}}\right)^i\vec{\varphi}=0\). \(k^{\prime}\) additionally generates solutions of equation \eqref{Fredholm Operator 1}, subject to conditions \eqref{Constraints 1}, provided that the \(\vec{k}_i^{\,\prime}\) are linearly independent.
\end{Corollary}
\begin{proof}
First observe \(\left\|\mathcal{A}\mathcal{P}_k\right\|+\left\|\varepsilon_{\eta}\mathcal{A}\tilde{\mathcal{P}}_{\eta}\right\|<1\) may always be satisfied (we have as an assumption \(\left\|\mathcal{A}\mathcal{P}_k\right\|<1\)) by choosing a small enough value of \(\left|\varepsilon_{\eta}\right|\). The second is merely by virtue of theorem \ref{theorem of solution}.
\end{proof}

Without loss of generality we may assume \(\varepsilon_{\eta}\geqslant 0\) and find an upper bound for \(\varepsilon_{\eta}\), thus
\[\varepsilon_{\eta}<\frac{1-\left\|\mathcal{A}\mathcal{P}_k\right\|}{\left\|\mathcal{A}\tilde{\mathcal{P}}_{\eta}\right\|}\]
\begin{Definition}
We define \(\varepsilon\leqslant\varepsilon_{\eta}\: \forall \: \left|\vec{\eta}_i\right|=1, \: \left<\vec{\eta}_i,\vec{y}_j\right>_2=\vec{0}\), which implies that
\[\varepsilon<\inf_{\left|\vec{\eta}_i\right|=1, \: \left<\vec{\eta}_i,\vec{y}_j\right>_2=\vec{0}}\varepsilon_{\eta}=\frac{1-\left\|\mathcal{A}\mathcal{P}_k\right\|}{\displaystyle\sup_{\left|\vec{\eta}_i\right|=1, \: \left<\vec{\eta}_i,\vec{y}_j\right>_2=\vec{0}}\left\|\mathcal{A}\tilde{\mathcal{P}}_{\eta}\right\|}\]
\end{Definition}
Since \(\sup_{\left|\vec{\eta}_i\right|=1, \: \left<\vec{\eta}_i,\vec{y}_j\right>_2=\vec{0}}\left\|\mathcal{A}\tilde{\mathcal{P}}_{\eta}\right\|\) is not dependent upon the choice of \(k\) it may be treated as a constant with regards to varying \(k\). Consequently, as \(k\) approaches any boundary of the open region of solutions, so \(\left\|\mathcal{A}\mathcal{P}_k\right\|\) approaches 1 from below since the region for which theorem \ref{solution disk theorem} remains valid can not extend beyond such a boundary.

\section{Discussion and Conclusions}
Whilst the approach identified in the preceding sections is restricted in that it is not concerned with finding the general solution of a particular equation, it is valuable in that it does provide a systematic approach, subject to several fairly standard conditions, of finding a particular solution to a particular equation. It is acknowledged that evidence is required of the method's utility and applicability in a larger class of general cases.

We have not given simple sufficient conditions for the existence of a \(k\) that offers a solution, but only guidance in finding such a \(k\) if it exists. Exploring how large a class of equations may be solved with this method is a promising direction for further work. The relevance and importance is provided by the linkage of the underlying equation to a class of models that may be used to describe the complex dynamics that are present within polymer micro-moulding processes (now increasingly important in supplying healthcare and related solutions) and a great many other physical phenomena.  

In conclusion, it does appear that we can sometimes solve equation \eqref{Fredholm Operator 1} subject to constraints \eqref{Constraints 1} by looking for values of \(k\) that minimise \(\left\|\mathcal{A}\mathcal{P}_k\right\|\) for some generalised notion of minimisation and then feeding these values of \(k\) back into \(\tilde{\mathcal{P}}_k\sum_{i=0}^{\infty}(\mathcal{A}\mathcal{P}_k)^i\vec{\varphi}\) which, if it evaluates to \(\vec{0}\) and provided that the \(\vec{k}_i\) are linearly independent, gives a solution \(\sum_{i=0}^{\infty}(\mathcal{A}\mathcal{P}_k)^i\vec{\varphi}\). We can do this even with a countably infinite number of constraints.

\bibliographystyle{elsarticle-num}
\bibliography{myPhDbib2}

\end{document}